\documentclass[submission]{FPSAC2021}

\newtheorem{thm}{Theorem}

\newtheorem{prop}[thm]{Proposition}
\newtheorem{remark}[thm]{Remark}
\newenvironment{rem}{\begin{remark}\normalfont}{\end{remark}}

\usepackage{lipsum}

\title[Muttalib--Borodin plane partitions and the hard edge of random matrix ensembles]{Muttalib--Borodin plane partitions and the hard edge of random matrix ensembles}

\author[D. Betea, A. Occelli]{Dan Betea\thanks{\href{mailto:dan.betea@gmail.com}{dan.betea@gmail.com}; partially supported by FWO Flanders project EOS 30889451.}\addressmark{1}, \and Alessandra Occelli\thanks{\href{mailto:alessandra.occelli@tecnico.ulisboa.pt}{alessandra.occelli@tecnico.ulisboa.pt}; partially supported by the HyLEF ERC starting grant 2016.} \addressmark{2}}

\address{\addressmark{1}Department of Mathematics, KU Leuven,  Leuven, Belgium. \\ \addressmark{2}Department of Mathematics, Instituto Superior T\'ecnico, Lisbon, Portugal.}

\received{\today}

\abstract{We study probabilistic and combinatorial aspects of natural volume-and-trace weighted plane partitions and their continuous analogues. We prove asymptotic limit laws for the largest parts of these ensembles in terms of new and known hard- and soft-edge distributions of random matrix theory. As a corollary we obtain an asymptotic transition between Gumbel and Tracy--Widom GUE fluctuations for the largest part of such plane partitions, with the continuous Bessel kernel providing the interpolation. We interpret our results in terms of two natural models of directed last passage percolation (LPP): a discrete $(\max, +)$ infinite-geometry model with rapidly decaying geometric weights, and a continuous $(\min, \cdot)$ model with power weights.}

\usepackage[backend=bibtex]{biblatex}
\usepackage{amsmath, amsthm, amssymb, verbatim, amsfonts, bbm, mathtools}
\usepackage{color}

\def\Z{\mathbb{Z}}
\def\N{\mathbb{N}}
\def\R{\mathbb{R}}
\def\P{\mathbb{P}}

\def\Id{\mathbbm{1}}
\def\Li{\mathrm{Li}_2}

\begin{document}

\maketitle

\section{Introduction}

\paragraph{Background.} Muttalib--Borodin (MB for short) ensembles are probability measures on $n$ real points $0 < x_1 < \dots < x_n$ of the from
\begin{equation} \label{eq:mb_def}
    \P(x_1 \in d x_1, \dots, x_n \in d x_n) = Z^{-1} \prod_{1 \leq i<j \leq n} (x_j-x_i)(x_j^\theta-x_i^\theta) \prod_{i=1}^n e^{-V(x_i)}
\end{equation}
where $\theta > 0$, $V$ is a potential and $Z$ is the normalization constant (partition function); they were introduced by Muttalib~\cite{mut95} as generalizations (if $\theta \ne 1$) of random matrix ensembles useful for studying disordered conductors. They are \emph{determinantal bi-orthogonal ensembles} with explicit correlation functions at least when $V$ is nice. Borodin explicitly computed a few examples~\cite{bor99} and further studied their asymptotic behavior at the ``edge'', i.e.~the behavior of $x_1$ as $n \to \infty$.

\paragraph{Main contribution.} In this paper we provide a combined algebraic-combinatorial and probabilistic perspective on such ensembles. We consider \emph{volume-and-trace dependent} simple distributions on plane partitions which give rise to discrete MB ensembles\footnote{Technically, these MB ensembles were first introduced in~\cite{fr05}.}---see Prop.~\ref{prop:disc_mb_dist}, and we interpret their largest parts/edge as certain last passage times in an infinite quadrant of rapidly decaying geometric random variables. The asymptotic behavior of these largest parts/LPP times has been previously encountered at the hard- and soft-edge of random matrix ensembles---see Thm.~\ref{thm:disc_lpp_he} and Thm.~\ref{thm:disc_lpp_tw}. In the simplest of such cases, for these last passage times and for the largest part of said plane partitions, we see a transition between the Gumbel distribution and the Tracy--Widom GUE distribution~\cite{tw94_airy} via the hard-edge random-matrix Bessel kernel~\cite{tw94_bessel}---see Remarks~\ref{rem:interpolation} and~\ref{rem:interpolation_2}. This result is similar to one of Johansson~\cite{joh08}. Furthermore, plane partitions give rise to a natural $q \to 1-$ limit ($q$ the volume parameter) where the discrete MB ensembles lead to a Jacobi-like continuous ensemble similar to Borodin's~\cite{bor99}. The smallest point in this ensemble has a natural $(\min, \cdot)$ LPP interpretation, and in studying its asymptotical distribution we recover a slight extension of Borodin's~\cite{bor99} probability distribution. See Theorems~\ref{thm:cont_lpp} and~\ref{thm:gap_prob_cont}. The latter has been shown by now to be universal for a wide range of potentials, see~\cite{mol20} and references therein. One of our main tool, principally specialized Schur processes~\cite{or03}, has been used on other occasions~\cite{fr05, bgs19} to bridge algebraic combinatorics and random matrix theory. Finally, let us note that Thm.~\ref{thm:disc_lpp_tw} is new, while a significantly expanded presentation of the other results will appear elsewhere~\cite{bo20}. 



\section{Main results}

\subsection{MB plane partitions and last passage percolation}

An \emph{$(M, N)$-based plane partition} $\Lambda$ is an array $\Lambda = (\Lambda_{i,j})_{1 \leq i \leq M, 1 \leq j \leq N}$ of non-negative integers satisfying $\Lambda_{i, j+1} \geq \Lambda_{i, j}$ and $\Lambda_{i+1, j} \geq \Lambda_{i, j}$ for all appropriate $i, j$. It can be viewed in 3D as a pile of cubes atop an $M \times N$ floor of a room (rectangle) where we place $\Lambda_{i, j}$ cubes above integer lattice point $(i, j)$ (starting from the ``back corner'' of the room). See Fig.~\ref{fig:pp_lpp_mb} (left) for an example. If $M=N=\infty$ we shall only speak of plane partitions, without a pre-qualifier (with the assumption that almost all $\Lambda_{i,j}=0$).  

Let us fix positive integer parameters $M \leq N$, possibly both equal to $\infty$. Fix also real parameters $0 \leq a, q \leq 1$ (not both 1) and $\eta, \theta \geq 0$. Denote $(Q, \tilde{Q}) = (q^\eta, q^\theta)$ for brevity. We consider the following distribution on $(M, N)$-based plane partitions:
\begin{equation} \label{eq:pp_measure}
    \P (\Lambda) = \frac{q^{\eta \text{ left vol}} \left[ a q^{\frac{\eta + \theta}{2}} \right]^{ \text{ central vol}} q^{\theta \text{ right vol}}}{Z} = \frac{Q^{\text{ left vol}} \left[ a \sqrt{Q \tilde{Q}} \right]^{ \text{ central vol}} \tilde{Q}^{\text{ right vol}}}{Z}
\end{equation}
where we call: \emph{central volume} (the word \emph{trace} is more customary in the literature, and it equals $\sum_i \Lambda_{i,i}$) the total number of cubes on the central slice of $\Lambda$ (marked in red in Fig.~\ref{fig:pp_lpp_mb} (left)); \emph{right volume} the number of cubes strictly to the right of the central slice (blue in fig.~cit.); and \emph{left volume} the number of cubes on the left (in green). Here $Z = \prod_{1 \leq i  \leq M} \prod_{1 \leq j \leq N} (1-aQ^{i-\frac{1}{2}} \tilde{Q}^{j-\frac{1}{2}})^{-1}$ is the partition (generating) function of all such plane partitions. 

We call this probability distribution the \emph{discrete Muttalib--Borodin} distribution on plane partitions---see below. If $a=1 = \eta = \theta$, it reduces to the usual $q^{\rm volume}$ distribution; if $q = 0$ it reduces to the $a^{\rm central\ vol} = a^{\rm trace}$ distribution (well-defined only for $M, N < \infty$).

\begin{figure} [!t]
    \begin{center}
        \includegraphics[scale=0.4]{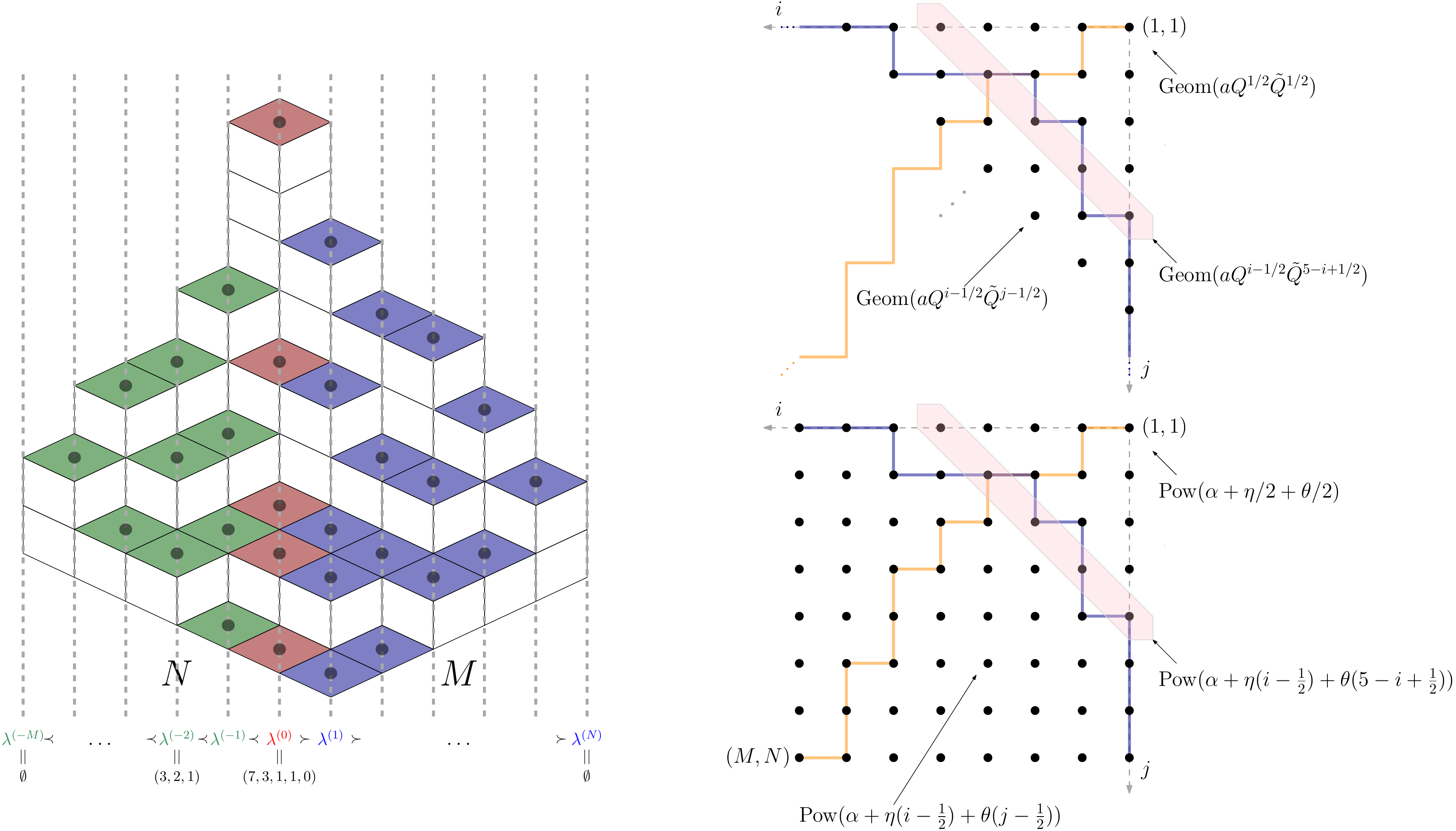}
    \end{center}
    \caption{
        Left: an $(M=5,N=6)$-based plane partition $\Lambda$ with color-coded left, center (red) and right cubes; the associated interlacing partitions are given at the bottom; the lozenges form a determinantal point process, each slice is a discrete MB ensemble, and slice $0$ contains $M$ points. Right top: the setting for the discrete geometric LPP we consider, and two polymers $\pi$ (orange) and $\varpi$ (blue) in~\eqref{eq:geo_lpp_def}; here $M=N=\infty$. Right bottom: the setting for the continuous power LPP we consider below with $M=N=8$.
    }
    \label{fig:pp_lpp_mb}
\end{figure}

Let us explain the naming for such distributions. Consider the standard identification of an $M \times N$-based plane partition $\Lambda$ with a sequence of ordinary interlacing partitions 
\begin{equation} \label{eq:disc_interlacing}
    \Lambda = (\emptyset = \lambda^{(0)} \prec \lambda^{(-M+1)} \prec \cdots \prec \lambda^{(0)} \succ \cdots \succ \lambda^{(N-1)} \succ \emptyset = \lambda^{(N)})
\end{equation}
(obtained by reading the heights of the horizontal lozenges on each vertical slice of Fig.~\ref{fig:pp_lpp_mb} (left)). We look at partition $\lambda^{(0)}$ and at its distribution. We choose this particular slice for simplicity only, looking at any other would yield similar formulas. Consider the point process $l^{(0)}$ with $M$ points given by $l^{(0)}_i = \lambda_i^{(0)} + M - i, 1 \leq i \leq M$ (the $M$ lozenges on the central slice of $\Lambda$ up to shift). We then the following proposition. 
Compare with~\eqref{eq:mb_def} and~\cite{fr05}.

\begin{prop} \label{prop:disc_mb_dist}
    Under the measure~\eqref{eq:pp_measure} for $M \leq N < \infty$, the $M$-point ensemble (slice) $l^{(0)}$ of $\Lambda$ has the following discrete Muttalib--Borodin distribution:
    \begin{equation} \label{eq:mbe_disc}
        \P(l^{(0)} = l) \propto \prod_{1 \leq i < j  \leq M} (Q^{l_j}-Q^{l_i}) (\tilde{Q}^{l_j} - \tilde{Q}^{l_i}) \prod_{1 \leq i \leq M} \left[a \sqrt{Q \tilde{Q}}\right]^{l_i} (\tilde{Q}^{l_i+1}; \tilde{Q})_{N-M}
    \end{equation}
    with $(x; u)_n = \prod_{0 \leq i < n} (1-xu^i)$ the $u$-Pochhammer symbol.
\end{prop}

Now we turn to introducing one of the last passage percolation models we consider. In the $M \times N$ integer rectangle (quadrant) consisting of points $(i, j)_{i, j \geq 1}$ with coordinates as in Fig.~\ref{fig:pp_lpp_mb} (top right) place at each point independent geometric random variables\footnote{$X$ is a \emph{geometric random variable} $X \sim {\rm Geom}(u)$ if $\P(X=k)=(1-u)u^k, \ k\in \N$.} $\omega^{\rm geo}_{i,j} \sim {\rm Geom}(a Q^{i-1/2} \tilde{Q}^{j-1/2})$. We look at the case $M=N=\infty$ but one could also consider these finite---see Fig.~\ref{fig:pp_lpp_mb} (top right). 

Let us look at the following \emph{last-passage times}:
\begin{equation} \label{eq:geo_lpp_def}
    L^{\rm geo}_1 = \max_{\pi} \sum_{(i,j) \in \pi} \omega^{\rm geo}_{i,j}, \quad L^{\rm geo}_2 = \max_{\varpi} \sum_{(i,j) \in \varpi} \omega^{\rm geo}_{i,j} 
\end{equation}
where $\pi$ is any down-left path from $(1,1)$ to $(M=\infty, N=\infty)$ (orange in Fig.~\ref{fig:pp_lpp_mb}) and $\varpi$ is any down-right path from $(M=\infty,1)$ to $(1, N=\infty)$ (blue in Fig.~\ref{fig:pp_lpp_mb}).
By Borel--Cantelli, only finitely many $\omega^{\rm geo}_{i,j}$'s are non-zero and so $L^{\rm geo}_i < \infty, i=1,2$ almost surely. Our first result is the following.

\begin{thm} \label{thm:disc_lpp_he}
    Let $M=N=\infty$. We have $L^{\rm geo}_1 = L^{\rm geo}_2 = \Lambda_{1,1}$ in distribution, where $\Lambda_{1,1}$ is the corner (largest) part of a Muttalib--Borodin-distributed plane partition $\Lambda$ as in~\eqref{eq:pp_measure}. Moreover, in the following $q,a \to 1-$ limit: 
    \begin{equation}
        q = e^{-\epsilon}, \quad a = e^{-\alpha \epsilon}, \quad \epsilon \to 0+ \quad (\alpha \geq 0 \text{ fixed})
    \end{equation}
    we have, for any $L \in \{ L^{\rm geo}_1, L^{\rm geo}_2, \Lambda_{1,1} \}$, that
    \begin{equation}
        \lim_{\epsilon \to 0+} \P \left( \epsilon L + \frac{\log (\epsilon \eta)}{\eta} + \frac{\log (\epsilon \theta)}{\theta} < s \right) = \det(1-\tilde{K}_{he})_{L^2(s, \infty)}
    \end{equation}
    where the RHS is a Fredholm determinant of the operator $\tilde{K}_{he} (x, y) = e^{-\frac{x}{2}-\frac{y}{2}} K_{he} (e^{-x}, e^{-y})$ and
    \begin{equation} \label{eq:he_kernel}
        K_{he} (x, y) = \frac{1}{\sqrt{xy}} \int\limits_{\delta + i \R} \frac{d \zeta}{2 \pi i} \int\limits_{-\delta + i \R} \frac{d \omega}{2 \pi i} \frac{F_{he}(\zeta)}{F_{he}(\omega)} \frac{x^\zeta}{y^\omega} \frac{1} {\zeta-\omega}, \quad F_{he}(\zeta) =\frac{\Gamma(\frac{\alpha} {2 \eta} - \frac{\zeta}{\eta} + \frac{1}{2})}{\Gamma(\frac{\alpha}{2 \theta} + \frac{\zeta}{\theta} + \frac{1}{2})}.
    \end{equation}
\end{thm}

\begin{rem}
    Let us make a definition and a few remarks on the above:
    \begin{itemize}
        \item The asymptotic distribution above, and most below, are \emph{Fredholm determinants}. To define them, recall first that an operator $K$ with kernel $K(x,y)$ acts on $L^2(X)$ ($X$ is an open interval all of our cases), e.g.~on functions $f \in L^2(X)$, via ``matrix multiplication'' $(Kf) (x) = \int_X K(x, y) f(y) dy$. If such an operator is trace-class---see e.g.~\cite{rom15}, the \emph{Fredholm determinant} of $1-K$ ($1$ the identity operator) on $L^2 (X)$ is
        defined by
        \begin{equation}
            \det(1-K)_{L^2(X)} = 1 + \sum_{m \geq 1} \frac{(-1)^m}{m!} \int_X \cdots \int_X \det_{1 \leq i, j \leq m} [K(x_i, x_j)] d x_1 \cdots d x_m
        \end{equation}
        where we put $m$ integrals in the $m$-th summand.
        \item Writing $\epsilon = 1/R$ for the limit result, $L$ has order $O(R \log R)$ and $O(R)$ fluctuations asymptotically. Contrast this with the $R^{1/3}$ limits of Johansson~\cite{joh00} and compare with similar (exponential) results of~\cite[Thm.~1.1a]{joh08}.
        \item The distributional equality $L^{\rm geo}_1 = L^{\rm geo}_2$ is not immediately obvious. Even if $\eta = \theta = 1$ (anti-diagonals $i+j=k+1$ are equi-distributed ${\rm Geom}(aq^k)$ random variables), $L_1$ is a maximum over random variables all of which are $\sum_{k\geq 1}{\rm Geom}(a q^k)$ (the distribution for this sum is furthermore explicit); $L_2$ does not enjoy this property.
        \item $K_{he}$ has the following hypergeometric-like form:
        \begin{equation}
            K_{ke}(x, y) =  \sum_{i,j=0}^{\infty} 
              \frac{(-1)^{i+j}  x^{\tfrac{\alpha-1}{2} + \eta (i + \tfrac12)} y^{\tfrac{\alpha-1}{2} + \theta (j + \tfrac12)} / [\alpha + \eta(i + \tfrac12) + \theta(j + \tfrac12)]  }{i! j! \Gamma \left( \tfrac{\alpha}{\theta} + (i + \tfrac12) \tfrac{\eta}{\theta} + \tfrac12 \right) \Gamma \left(\frac{\alpha}{\eta} + \tfrac12 + (j + \tfrac12) \tfrac{\theta}{\eta} \right) }.
        \end{equation}
    \end{itemize}
\end{rem}



\begin{rem} \label{rem:interpolation}
    Consider (again) the \emph{equi-distributed-by-diagonal} case $\eta = \theta = 1$ (i.e.~anti-diagonal $i+j=k+1$ has $k$ iid ${\rm Geom} (a q^k)$ random variables on it). We have
    \begin{equation}
        \begin{split}
        \tilde{K}_{he} (x, y) &= e^{-\frac{x}{2}-\frac{y}{2}} K_{\alpha, \rm Bessel} (e^{-x}, e^{-y}) \\
                              &= \int\limits_{-\delta+i\R} \frac{d \omega}{2 \pi i} \int\limits_{\delta+i\R} \frac{d \zeta}{2 \pi i}  \frac{\Gamma(\tfrac{\alpha}{2} + \tfrac12 - \zeta)}{\Gamma(\tfrac{\alpha}{2} + \tfrac12 + \zeta)} \frac{\Gamma(\tfrac{\alpha}{2} + \tfrac12 + \omega)} {\Gamma(\tfrac{\alpha}{2} + \tfrac12 - \omega)} \frac{e^{-x\zeta}}{e^{-y\omega}}\frac{1}{\zeta-\omega} 
        \end{split}
    \end{equation}
    ($0 < \delta < \tfrac12$) with $K_{\alpha, \rm Bessel}$ the random matrix hard-edge continuous Bessel kernel~\cite{tw94_bessel}
    \begin{equation} \label{eq:Bessel_kernel}
        K_{\alpha, \rm Bessel} (x, y) = \int_{0}^1 J_{\alpha} (2 \sqrt{ux}) J_{\alpha} (2 \sqrt{uy}) du = \frac{J_{\alpha} (\sqrt{x}) \sqrt{y} J'_{\alpha} (\sqrt{y}) - \sqrt{x} J'_{\alpha} (\sqrt{x}) J_{\alpha} (\sqrt{y}) } {2(x-y)}
    \end{equation} 
    with the $J$'s Bessel functions. Let us write $F_{\alpha}(s) = \det(1-\tilde{K}_{he})_{L^2(s, \infty)}$. We then have, from Johansson~\cite{joh08}, the following Gumbel to Tracy--Widom interpolation property:
    \begin{itemize}
        \item $\lim_{\alpha \to 0} F_{\alpha}(s) = F_0(s) = e^{-e^{-s}}$ with the latter the Gumbel distribution;
        \item $\lim_{\alpha \to \infty} F_{\alpha}(-2 \log(2(\alpha-1)) + (\alpha-1)^{-2/3} s) = F_{\rm TW}(s)$ with the latter the Tracy--Widom GUE distribution~\cite{tw94_airy}.
    \end{itemize}
\end{rem}

Neither the Gumbel nor Tracy--Widom distributions appearing above are surprising. Indeed the first is the asymptotic distribution of the largest part of a $q^{\rm volume}$-distributed plane partition~\cite[Thm.~1]{vy06} (our case with $a=1$). To see Tracy--Widom GUE fluctuations directly, consider the result below. What is remarkable nonetheless is the interpolation/transition property of the continuous Bessel kernel in ``exponential'' variables between Gumbel (``universal'' asymptotic maximum of iid random variables) and Tracy--Widom GUE (asymptotic maximum of correlated systems like eigenvalues of Hermitian random matrices).

\begin{thm} \label{thm:disc_lpp_tw}
    Let $M=N=\infty$, and let $0 < a < 1$ be fixed. In the following $q =e^{-\epsilon} \to 1-$ as $\epsilon \to 0+$ limit and for any $L \in \{ L^{\rm geo}_1, L^{\rm geo}_2, \Lambda_{1,1} \}$ as in Thm.~\ref{thm:disc_lpp_he}, we have:
    \begin{equation}
        \lim_{\epsilon \to 0+} \P \left( \frac{L - c_1 \epsilon^{-1}} {c_2 \epsilon^{-1/3}} < s \right) = F_{\rm TW} (s)
    \end{equation}
    where $c_1, c_2 \in \R_+$ are explicit\footnote{\label{ft:asymptorics} Let $b=\sqrt{a}, z_c = \tfrac{b(\theta-\eta)+\sqrt{4\eta\theta+b^2(\theta-\eta)^2}}{2 \theta}, v_c = -\eta^{-1} \log(1-b z_c)-\theta^{-1} \log(1-b/z_c)$ and $S(z, v) = \eta^{-1} \Li(bz) - \theta^{-1} \Li(b/z) - v \log(z)$. We have $c_1 = y_c$ and $c_2 = \left( 2^{-1}(z \partial_z)^3 S |_{z=z_c,v=v_c} \right)^{1/3}$. Here $\Li$ is the dilogarithm function \url{https://fr.wikipedia.org/wiki/Dilogarithme}.} and with $F_{\rm TW}$ the Tracy--Widom GUE distribution.  
\end{thm}

\begin{rem} \label{rem:interpolation_2}
    Let us summarize the three asymptotic regimes when $\eta = \theta = 1$. As $q \to 1-$, $L$ as in Theorems~\ref{thm:disc_lpp_he} and~\ref{thm:disc_lpp_tw} has: 
    \begin{itemize}
        \item Gumbel fluctuations, if $a=1$;
        \item Tracy--Widom fluctuations (on a different scale), if $0 < a < 1$ fixed;
        \item transitional (exponential) hard-edge Bessel fluctuations, if $a \to 1$ critically.
    \end{itemize} 
\end{rem}

\subsection{Continuous MB ensembles and $(\min, \cdot)$ last passage percolation}

In this section we use continuous parameters $\alpha, \eta, \theta \geq 0$ (not all zero) and integer parameters $M \leq N < \infty$ (same as above, except now we keep them finite at the beginning). 

On the lattice $(i, j)_{1 \leq i \leq M, 1 \leq j \leq N}$ place, at $(i, j)$, independent power random variables\footnote{$Y$ is a \emph{power random variable} $Y \sim {\rm Pow}(\beta)$ if $\P (Y \in dx) = \beta x^{\beta-1}, \ x \in [0,1]$, for $\beta > 0$.} $\omega^{\rm pow}_{i,j} \sim {\rm Pow} (\alpha + \eta(i - \tfrac12) + \theta (j-\tfrac12))$. Let
\begin{equation} \label{eq:pow_lpp_def}
    L^{\rm pow}_1 = \min_{\pi} \prod_{(i,j) \in \pi} \omega^{\rm pow}_{i,j}, \quad L^{\rm pow}_2 = \min_{\varpi} \prod_{(i,j) \in \varpi} \omega^{\rm pow}_{i,j} 
\end{equation}
where $\pi$ is any down-left path from $(1,1)$ to $(M, N)$ (orange in Fig.~\ref{fig:pp_lpp_mb} (bottom right)) and $\varpi$ is any down-right path from $(M,1)$ to $(1, N)$ (blue in fig.~cit.).


We have the following finite $M,N$ result. Note again the first equality in distribution is not immediately obvious. Part of it was anticipated, up to change of variables, in~\cite{fr05}.

\begin{thm}\label{thm:cont_lpp}
    We have $L^{\rm pow}_1 = L^{\rm pow}_2 = x_1$ in distribution, with $x_1$ the smallest (hard-edge\footnote{The name hard-edge stands for the fact that 0 is a ``hard edge'' of the support of the distribution; no number can go below 0.}) point in the following Muttalib--Borodin distribution on $M$-point ensembles $\vec{x} = (0 < x_1 < \dots < x_M < 1)$:
    \begin{equation} \label{eq:mbe_cont}
        \P(\vec{x} \in d \vec{x})  \propto \prod_{1 \leq i < j  \leq M} (x_j^\eta-x_i^\eta) (x_j^\theta-x_i^\theta) \prod_{1 \leq i \leq M} x_i^{\alpha + \frac{\eta + \theta}{2} - 1} (1-x_i^\theta)^{N-M}.
    \end{equation}
\end{thm}

If $\eta = \theta = 1$, the above is an example of the Jacobi random matrix ensemble. 

Finally, our next result is asymptotic. We take $M,N \to \infty$, and we can even do this independently.

\begin{thm} \label{thm:gap_prob_cont}
    With $L \in \{L^{\rm pow}_1, L^{\rm pow}_2, x_1 \}$ as in Thm.~\ref{thm:cont_lpp} and $K_{he} (x, y)$ as in~\eqref{eq:he_kernel} we have:
    \begin{equation}
        \begin{split}
        \lim_{M, N \to \infty} \P \left( \frac{  L  } {M^{\frac{1}{\eta}} N^{\frac{1}{\theta}}} < r \right) = \det (1 - K_{he})_{L^2(0,r)}
        \end{split}.
    \end{equation} 
\end{thm}

The terminology ``hard edge'' now becomes clear. We are looking at $L$ close to 0, the hard-edge of the support $[0,1]$ for the ensemble in~\eqref{eq:mbe_cont}. When $\eta = \theta = 1$, $K_{he}$ is the hard-edge Bessel kernel~\cite{tw94_bessel} (scaling of the Laguerre or Jacobi ensembles around 0). When $\eta = 1$, $K_{he}$ is Borodin's~\cite{bor99} generalization of the Bessel kernel, appearing in the scaling of various Muttalib--Borodin ensembles---see e.g.~\cite{mol20} and references therein.

\section{Sketches of proofs}

\begin{proof}[Proof of Prop.~\ref{prop:disc_mb_dist}]
Muttalib--Borodin-distributed plane partitions $\Lambda$, under the identification~\eqref{eq:disc_interlacing}, are Schur processes~\cite{or03}. In our case this means the measure~\eqref{eq:pp_measure} can be written as
\begin{equation} \label{eq:schur_proc}
    \P (\Lambda) = Z^{-1} \prod_{i=0}^{M-1} s_{\lambda^{(-i)} / \lambda^{(-i-1)}} (\sqrt{a} Q^{i+1/2}) \prod_{i=0}^{N-1} s_{\lambda^{(i)} / \lambda^{(i+1)}} (\sqrt{a} \tilde{Q}^{i+1/2})
\end{equation}
with the partition (generating) function $Z = \prod_{i=1}^M \prod_{j=1}^N (1-a Q^{i-1/2} \tilde{Q}^{j-1/2})^{-1}$; with $(Q, \tilde{Q}) = (q^\eta, q^\theta)$ as before; and with $s_{\lambda / \mu}$ the skew Schur polynomials (functions)~\cite[Ch.~I.5]{mac}. These latter, evaluated in one variable, contribute the right amount to the measure: $s_{\lambda / \mu} (x) = x^{|\lambda| - |\mu|} \Id_{\mu \prec \lambda}$ by observing ${\rm left\ vol} = \sum_{i=-M}^{-1} |\lambda^{(i)}|$, ${\rm central\ vol} = |\lambda^{(0)}|$ and ${\rm right\ vol} = \sum_{i=1}^{N} |\lambda^{(i)}|$.

As such, the marginal distribution of $\lambda^{(0)}$ is a Schur measure~\cite{oko01, or03}. We obtain, after some simplification:
\begin{equation}
    \P (\lambda^{(0)} = \lambda) = Z^{-1} a^{|\lambda|} Q^{|\lambda|/2} \tilde{Q}^{|\lambda|/2} s_{\lambda} (1, Q, \dots, Q^{M-1}) s_{\lambda} (1, \tilde{Q}, \dots, \tilde{Q}^{N-1})
\end{equation}
with $Z$ as before and with $s_{\lambda}$ the regular Schur polynomials ($s_\lambda = s_{\lambda/\emptyset}$). To finish, let us first notice that $\ell(\lambda) \leq M$ from the interlacing constraints~\eqref{eq:disc_interlacing}. Moreover, specializing Schur polynomials in a geometric progression (the \emph{principal specialization}) is explicit~\cite[Ch.~I.3]{mac}:
$s_{\lambda}(1, u, \dots, u^{n-1}) = \prod_{1 \leq i < j \leq n} \frac{u^{\lambda_i + M -i} - u^{\lambda_j + M - j}}{u^{M -i} - u^{M - j}}$. Recalling $l_i = \lambda_i + M - i$ and so $l_i = M-i$ for $i > M$, we see the length $N$ Vandermonde-like product in one of the terms above can be rewritten as one of length $M$ plus additional univariate factors as stated. Note we gauge away constants independent of the $l_i$'s.
\end{proof}

\begin{proof}[Proof of Thm.~\ref{thm:disc_lpp_he}]
There are two parts of the statement. For the finite part, consider the array of numbers $(\omega^{\rm geo}_{i, j})_{1 \leq i \leq M, 1 \leq j \leq N}$ as considered but first with $M \leq N$ both finite. We can transform this array, bijectively, into a plane partition $\Lambda$ via both row insertion Robinson--Schensted--Knuth (RSK)~\cite{knu70} and column insertion RSK (Burge)~\cite{bur74} algorithms. In both cases if we start with distribution $\omega^{\rm geo}_{i,j} \sim {\rm Geom}(a Q^{i-\frac{1}{2}} \tilde{Q}^{j-\frac{1}{2}})$, we end up with $\Lambda$ distributed as in~\eqref{eq:pp_measure}---see~\cite{bbbccv18} and references therein. Now the Greene--Krattenthaler~\cite{gre74, kra06} theorem states that $L^{\rm geo}_1 = \Lambda_{1,1}$ (for row RSK) and $L^{\rm geo}_2 = \Lambda_{1,1}$ (for column RSK). This implies that $L^{\rm geo}_1 = L^{\rm geo}_2 = \Lambda_{1,1}$ in distribution, for $M, N$ finite. For $M = N = \infty$ we just observe that almost surely only finitely many $\omega^{\rm geo}_{i,j}$ will be non-zero by Borel--Cantelli and the results just described go through.

For the second part, the previous proof implies that in distribution $L^{\rm geo}_1 = L^{\rm geo}_2 = \Lambda_{1,1} = \lambda_1$ where the last quantity is the first part of a random partition distributed as
\begin{equation}
    \P (\lambda) = Z^{-1} a^{|\lambda|} Q^{|\lambda|/2} \tilde{Q}^{|\lambda|/2} s_{\lambda} (1, Q, Q^2, \dots) s_{\lambda} (1, \tilde{Q}, \tilde{Q}^2,\dots )
\end{equation}
(both specializations are now infinite geometric series as $M=N=\infty$) with $Z = \prod_{i=1}^\infty \prod_{j=1}^\infty (1-a Q^{i-1/2} \tilde{Q}^{j-1/2})^{-1}$. This is again a Schur measure, and it is determinantal~\cite{oko01}. Namely, the point process $\{ \lambda_i - i + \frac{1}{2} | i \geq 1\}$ is a determinantal point process: i.e.~if $m \geq 1$ and $k_1, \dots, k_m \in \Z+\frac{1}{2}$ we have:
\begin{equation}
    \P (\{ k_1, \dots, k_m \} \in \{ \lambda_i - i + \tfrac{1}{2} | i \geq 1 \}) = \det_{1 \leq i, j \leq m} K_d(k_i, k_j) 
\end{equation}
where the discrete ($\ell^2$ operator) kernel $K_d$ equals (for $\delta > 0$ very small)
\begin{equation} \label{eq:K_d}
    K_d(k, \ell) = \oint\limits_{|w| = 1 - \delta} \frac{dw}{2 \pi i w} \oint\limits_{|z| = 1 + \delta} \frac{dz} {2 \pi i z} \frac{F_d(s)}{F_d(w)}  \frac{w^{\ell}}{z^{k}} \frac{\sqrt{zw}}{z-w}, \quad F_d(z) = \frac{(\sqrt{a} \tilde{Q}^{1/2}/z; \tilde{Q})_{N}}{(\sqrt{a} Q^{1/2}z; Q)_{M}}
\end{equation}
with $M=N=\infty$ (we nonetheless record the formula for arbitrary $M,N$ for use later). The combinatorial meaning of the integral is coefficient extraction: $K_d(k, \ell)$ is the coefficient of $z^k/w^\ell$ in the generating series above. The condition $|z| > |w|$ makes the formula true analytically as well.

Inclusion-exclusion yields that the distribution of $L \in \{L^{\rm geo}_1, L^{\rm geo}_2, \Lambda_{1,1}, \lambda_1\}$ is the discrete Fredholm determinant of $K_d$: $\P(L \leq l) = \det(1-K_d)_{\ell^2 \{l+1/2,l+3/2,\dots\}}$. 

To finish the proof, we still have to show $\det(1-K_d) \to \det(1-\tilde{K}_{he})$ in the limit $l = \frac{s}{\epsilon} - \frac{\log(\epsilon \eta)}{\epsilon \eta} - \frac{\log(\epsilon \theta)}{\epsilon \theta}$ as $\epsilon \to 0+$. The first step is to show that $\epsilon^{-1} K_d(k, \ell) \to K_{he} (x, y)$ for $(k, \ell) = \frac{(x, y)}{\epsilon}  - \frac{\log(\epsilon \eta)}{\epsilon \eta} - \frac{\log(\epsilon \theta)}{\epsilon \theta}$; the second to show convergence of Fredholm determinants. Both steps require some analytic justification of interchanging of limits, integrals, and sums (the defining series for a Fredholm determinant). Modulo these details which we omit for brevity, to show $\epsilon^{-1} K_d(k, \ell) \to K_{he} (x, y)$ one simply uses the limiting relation
\begin{equation}
    \log (u^c; u)_{\infty} = -\frac{\pi^2}{6} r^{-1} + \left( \frac{1}{2} - c \right) \log r + \frac{1}{2} \log(2 \pi) - \log \Gamma(c) + O(r)
\end{equation}
where $u = e^{-r} \in \{Q, \tilde{Q}\}, r \to 0+, c\notin -\N$, together with a change of variables $(z, w) = (e^{\epsilon \zeta}, e^{\epsilon \omega})$ in~\eqref{eq:K_d}. The contours transform appropriately and the double integral~\eqref{eq:K_d} becomes~\eqref{eq:he_kernel} in the limit $\epsilon \to 0+$.
\end{proof}

\begin{proof}[Proof of Thm.~\ref{thm:cont_lpp}]
    The proof is a $q \to 1-$ limit of the argument above. Let us take $M \leq N < \infty$ and $\alpha \geq 0$ fixed. In the limit 
    \begin{equation}
        q = e^{-\epsilon}, \quad a = e^{- \alpha \epsilon}, \quad \lambda_i = -\epsilon^{-1} \log x_i, \quad \epsilon \to 0+
    \end{equation}
    the process $\Lambda$ from~\eqref{eq:disc_interlacing} converges, in the sense of finite dimensional distributions, to a continuous process ${\bf X}$ of corresponding interlacing vectors with elements in $(0, 1)$ almost surely. Importantly, the slice $\lambda^{(0)} = \lambda$ of $\Lambda$ converges to an ensemble we call $\vec{x}$ of $M$ points $0 < x_1 < \cdots < x_M < 1$ with distribution given by the $q \to 1-$ limit of~\eqref{eq:mbe_disc}; this is the stated distribution from~\eqref{eq:mbe_cont}. We see this using simple limits like: $Q^{l_i} \to x_i^{\eta}, \tilde{Q}^{l_i} \to x_i^{\theta}, a^{l_i} \to x_i^{\alpha}$ and finally $(\tilde{Q}^{l_i+1}; \tilde{Q})_{N-M} \to (1-x^\theta)^{N-M}$.  

    That $L^{\rm pow}_1 = L^{\rm pow}_2 = x_1$ in distribution comes from the fact that, with the setup from the beginning of the proof of Thm.~\ref{thm:disc_lpp_he} (keeping $M,N$ finite), we have $L^{\rm geo}_1 = L^{\rm geo}_2 = \lambda_1$. Moreover the corresponding geometric random variables converge to power random variables: $\exp( -\epsilon \omega_{i, j}^{\rm geo}) \to \omega_{i, j}^{\rm pow}$. Then 
    \begin{equation}
        \max \sum \omega_{i, j}^{\rm geo} = \max \left( - \sum \epsilon^{-1} \log \omega_{i, j}^{\rm pow} \right) = - \epsilon^{-1} \log \min \left( \prod \omega_{i, j}^{\rm pow} \right)
    \end{equation} 
    (with sums/products being over the appropriate sets of directed paths) showing that $\exp(- \epsilon L_i^{\rm geo} ) \to L_i^{\rm pow}, i=1,2$. Together with the fact that $\exp(-\epsilon \lambda_1) \to x_1$ and the discrete finite $M, N$ Greene--Krattenthaler Theorem~\cite{gre74, kra06}, this finishes the proof.
\end{proof}

\begin{proof}[Proof of Thm.~\ref{thm:gap_prob_cont}]
    The ensemble $\vec{x}$ from Theorem~\ref{thm:cont_lpp} is determinantal with kernel $K_c$, as a limit $\epsilon \to 0+$ of the ensemble $q^{l}$ with (recall) $l_i = \lambda_i+M-i,1\leq i \leq M$. The kernel $K_c$ comes from $\epsilon^{-1} K_d(k, \ell) \to K_c(x, y)$ with $K_d$ as in~\eqref{eq:K_d}, with $M \leq N$ finite, with $a = e^{-\alpha \epsilon}$, $(k, \ell) = -\epsilon^{-1} (\log x, \log y)$, and with changing the variables $(z, w) = (e^{\epsilon \zeta}, e^{\epsilon \omega})$ inside the integral to have a finite limit. We have
    \begin{equation}
        K_c(x, y) = \frac{1}{\sqrt{xy}} \int\limits_{-\delta+i \R} \frac{d \omega}{2 \pi i} \int\limits_{\delta+i \R} \frac{d \zeta}{2 \pi i} \frac{F_c(\zeta)}{F_c(\omega)} \frac{x^\zeta}{y^\omega} \frac{1}{\zeta-\omega}, \quad F_c(\zeta) = \frac{(\frac{\alpha}{2 \theta} + \frac{\zeta}{\theta} + \frac{1}{2})_{N} } {(\frac{\alpha}{2 \eta} - \frac{\zeta}{\eta} + \frac{1}{2})_{M}}
    \end{equation}
    where $(a)_n = \prod_{1 \leq i < n} (a+i) = \Gamma(a+n)/\Gamma(a)$ is the Pochhammer symbol. Finally $ M^{-1/\eta} N^{-1/\theta} K_c(x M^{-1/\eta} N^{-1/\theta}, y M^{-1/\eta} N^{-1/\theta}) \to K_{he} (x, y)$ as $M, N \to \infty$ from Stirling's approximation of the Gamma function; the contours remain unchanged in the limit; and further estimates show Fredholm determinants converge to Fredholm determinants proving the result.
\end{proof}

\begin{proof}[Proof of Thm.~\ref{thm:disc_lpp_tw}]
    The argument is similar to the asymptotical part of the proof of Thm.~\ref{thm:disc_lpp_he}, but the details get more complicated. Let us write $R = \epsilon^{-1} \to \infty$. The bulk of the argument is showing that, with $M=N=\infty$ and as $R \to \infty$ we have $R^{1/3} K_d(k, \ell) \to A(x, y)$ with $K_d$ as in~\eqref{eq:K_d} and for $(k, \ell) = c_1 R + (x, y) c_2 R^{1/3}$. Here $A$ is the Airy kernel~\cite{tw94_airy} given by 
    \begin{equation}
        A(x,y) = \int_{-\delta + i \R} \frac{d \omega}{2 \pi i} \int_{\delta + i \R} \frac{d \zeta}{2 \pi i} \frac{\exp(-x \zeta + \zeta^3/3)}{\exp(-y \omega + \omega^3/3)} \frac{1}{\zeta - \omega}
    \end{equation}
    and we recall $F_{\rm TW}(s) = \det(1-A)_{L^2(s, \infty)}$. Some extra estimates then are needed to show the gap probability $\P(L \leq l) = \det(1-K_d)_{\ell^2(l+1/2,\dots)} \to \det (1-A)_{L^2\{s, \infty\}}$ when $l = c_1 R + s c_2 R^{1/3}$ and $R \to \infty$. The constants $c_1, c_2$ are given in footnote~\ref{ft:asymptorics}.

    We begin by taking $0 < b=\sqrt{a} < 1$ fixed. We note the asymptotic estimate $(g u; u)_{\infty} \approx -r^{-1} \Li(g)$ if $g$ is away from $0$ and $1$ and $u=e^{-r}, r \to 0+$ ($\Li$ the dilogarithm). In our case $u \in \{Q, \tilde{Q}\}$ and we can then estimate $F_d(z)/F_d(w)$ in~\eqref{eq:K_d}. It follows that $K_d(k, \ell) \approx \oint \oint e^{R [S(z)-S(w)]} \frac{dzdw}{z-w}$ where $S(z) = \eta^{-1} \Li(bz) - \theta^{-1} \Li(b/z) - (k/R) \log z$ (for $S(w)$, $k \mapsto \ell$). Let both $k,\ell \approx v R$ and moreover take $v = v_c = -\eta^{-1} \log(1-b z_c)-\theta^{-1} \log(1-b/z_c), z_c = \tfrac{b(\theta-\eta)+\sqrt{4\eta\theta+b^2(\theta-\eta)^2}}{2 \theta}$. First note $c_1 = v_c$ by definition. Moreover, in this case, $S' = S'' = 0$ at $z=z_c$ ($'=\frac{d}{dz}$) and the asymptotic contribution of $\oint \oint e^{R [S(z)-S(w)]} \frac{dzdw}{z-w}$ comes from the third derivative $S'''$. We Taylor-expand $S(z)$ around $z = z_c (1 + \zeta R^{-1/3}/c_2)$ and $k = v_c R + x c_2 R^{1/3}$ in powers of $R^{-1/3}$, and likewise for $S(w)$ with $(\omega, \ell, y)$ replacing $(\zeta, k, x)$. The first few terms (for $z$) are
    \begin{equation}
        S(z_c) + \frac{S^{(3)} (z_c)} {6} \frac{\zeta^{3}}{c_2^3} - x \zeta + O(R^{-1/3}) = S(z_c) + \frac{\zeta^{3}}{3} - x \zeta + O(R^{-1/3})
    \end{equation}
    having chosen $c_2$ so that we have the simpler expansion on the right. Note also that the contours become the vertical lines given in the definition of the Airy kernel above. Modulo some extra estimates omitted here we have:
    \begin{equation}
        R^{1/3} K_d(k, \ell) \approx R^{1/3} \oint \oint \frac{e^{R [S(z)-S(w)]}}{z-w} \frac{dzdw}{(2 \pi i)^2}  \approx \frac{1}{(2 \pi i)^2} \int_{C_{\zeta}} \int_{C_{\omega}} \frac{e^{\frac{\zeta^3}{3} - x \zeta} }{e^{\frac{\omega^3}{3} - y \omega}} \frac{d \zeta d \omega}{\zeta - \omega}
    \end{equation}
    (with $(C_\zeta, C_\omega) = (\delta, -\delta) + i \R, \delta > 0$) showing $R^{1/3} K_d(k, \ell) \to A(x, y)$ as desired.
\end{proof}


\printbibliography

\end{document}